\documentclass[11pt,a4paper,reqno]{amsart}
\usepackage{a4wide}
\usepackage{amssymb,amsmath,amsthm,mathrsfs}
\usepackage{graphicx}
\usepackage{float}
\usepackage{colortbl}
\usepackage{enumerate}
\usepackage{upref}
\usepackage{stackrel}
\usepackage[bookmarks=false]{hyperref}
\usepackage[T2A,T1]{fontenc}
\DeclareSymbolFont{cyrillic}{T2A}{cmr}{m}{n}
\DeclareMathSymbol{\D}{\mathalpha}{cyrillic}{196}

\newtheorem{theorem}{Theorem}[section]
\newtheorem{lemma}[theorem]{Lemma}

\theoremstyle{definition}

\theoremstyle{remark}
\newtheorem{remark}[theorem]{Remark}

\numberwithin{equation}{section}

\newtheorem*{condition}{Condition}

\usepackage{enumitem}
\makeatletter
\def\namedlabel#1#2{\begingroup
   #2%
 \def\@currentlabel{#2}%
   \phantomsection\label{#1}\endgroup
}

\def\R{\ensuremath{\mathbb R}}
\def\N{\ensuremath{\mathbb N}}

\def\I{\ensuremath{{\bf 1}}}
\def\e{{\ensuremath{\rm e}}}

\def\L{\ensuremath{\mathcal L}}

\def\p{\ensuremath{\mathbb P}}

\def\A{\ensuremath{A^{(q)}}}

\def\Y{\mathcal{Y}}
\def\a{\alpha}
\def\1{{\bf 1}}

\def\ie{{\em i.e.}, }

\def\dist{\ensuremath{\text{dist}}}

\def\TF{\mathcal{T}}

\def\eps{\varepsilon}

\numberwithin{equation}{section}

\begin{document}

\title[EVL for sequences of intermittent maps]{Extreme Value Laws for sequences of intermittent maps}


%    author one information

\author[A. C. M. Freitas]{Ana Cristina Moreira Freitas}
\address{Ana Cristina Moreira Freitas\\ Centro de Matem\'{a}tica \&
Faculdade de Economia da Universidade do Porto\\ Rua Dr. Roberto Frias \\
4200-464 Porto\\ Portugal} \email{\href{mailto:amoreira@fep.up.pt}{amoreira@fep.up.pt}}
\urladdr{\url{http://www.fep.up.pt/docentes/amoreira/}}

\author[J. M. Freitas]{Jorge Milhazes Freitas}
\address{Jorge Milhazes Freitas\\ Centro de Matem\'{a}tica \& Faculdade de Ci\^encias da Universidade do Porto\\ Rua do
Campo Alegre 687\\ 4169-007 Porto\\ Portugal}
\email{\href{mailto:jmfreita@fc.up.pt}{jmfreita@fc.up.pt}}
\urladdr{\url{http://www.fc.up.pt/pessoas/jmfreita/}}

\author[S. Vaienti]{Sandro Vaienti}
\address{Sandro Vaienti\\ Aix Marseille Universit\'e, CNRS, CPT, UMR 7332
\\ 13288 Marseille, France and
Universit\'e de Toulon, CNRS, CPT, UMR 7332, 83957 La Garde, France.}
\email{vaienti@cpt.univ-mrs.fr}
\urladdr{\url{http://www.cpt.univ-mrs.fr/~vaienti/}}

%\author{}
%\address{}
%\curraddr{}
%\email{}
%\thanks{}
%
%%    author two information
%\author{}
%\address{}
%\curraddr{}
%\email{}
%\thanks{}
%
%    \subjclass is required.
%\subjclass[2010]{Primary }

\date{\today}

%\dedicatory{}
%
%%    "Communicated by" -- provide editor's name; required.
%\commby{}
%
%    Abstract is required.
\begin{abstract}
We study non-stationary stochastic processes arising from sequential dynamical systems built on maps with a neutral fixed points and prove the existence of Extreme Value Laws for such processes. We use an approach developed in \cite{FFV16}, where we  generalised the theory of extreme values for non-stationary stochastic processes, mostly  by weakening the uniform mixing condition that was previously  used in this setting. The present work is an extension of our previous results for concatenations of uniformly expanding maps obtained in \cite{FFV16}.  

\end{abstract}

\maketitle

\section{Introduction}

The erratic behaviour of chaotic dynamical systems motivated the use of probabilistic tools to study the statistical behaviour of such systems. The time evolution of chaotic systems gives rise to time series resulting from evaluating an observable function along the orbits of the system. 

The mixing features of the systems determine the dependence structure of the processes, leading, usually, to some sort of asymptotic independence that, often, allows to recover the behaviour of purely random, independent and identically distributed sequences of random variables.

The ergodic properties of the systems are tied to the existence of invariant measures, which endow the stochastic processes arising from such systems with stationarity. In some sense, the invariant measures, which usually have some physical significance, determine the system itself. However, sometimes the exact formula for the invariant measure is not accessible and one has to rely on reference measures with respect to which these processes are not stationary anymore.

Relaxing stationarity gives rise to non-autonomous dynamical systems for which the study of limit theorems is just at the beginning.
Here, we will focus on the particular problem of studying the existence of limiting Extreme Value Laws (EVL), which, as shown in \cite{FFT10}, is related to the occurrence of rare events and the study of Hitting and Return Time Statistics.

The study of the extremal properties of non-stationary stochastic processes was introduced by H\"usler  in \cite{H83,H86} and the theory was built up on this initial effort, which generalised Leadbetter's conditions and approach to deal with general stationary stochastic processes. This fact precluded its application in a dynamical setting. In \cite{FFV16}, the authors developed a more general theory, based on necessary adjustments to Leadbetter's conditions and a much more refined way of dealing with clustering, originally developed in \cite{FFT12, FFT15}, which, ultimately, allowed the application to non-autonomous dynamical systems.

We will use the theory established in \cite{FFV16} to study \emph{sequential dynamical systems} originated by the composition of intermittent maps.  Sequential dynamical systems were introduced by Berend and Bergelson \cite{BB84}, as a non-stationary system in which a  concatenation of maps is applied to a given point in the underlying space, and the probability is taken as  a conformal measure, which allows the use of the transfer operator (Perron-Fr\"obenius) as a useful tool to quantify the loss of memory of any prescribed initial observable. The theory of sequential systems was later developed  in the fundamental paper by Conze and Raugi \cite{CR07}, where a few limit theorems, in particular the Central Limit Theorem, were proved for concatenations of  one-dimensional dynamical systems, each possessing a transfer operator with a quasi-compact structure on a suitable Banach space. For the same systems and others, even in higher dimensions, the Almost Sure Invariance Principle was subsequently shown in \cite{HNTV17}.

Both papers \cite{CR07, HNTV17} dealt however with uniformly expanding maps, for which the transfer operators admits a spectral gap and the correlations decays exponentially. In a different direction, a class of sequential systems given by composition of non-uniformly expanding maps of  Pomeau-Manneville type was studied in \cite{AHNTV15},  by perturbing the slope at the indifferent fixed point $0.$ Polynomial decay of correlations was proved for particular classes of centred observables, which could also be interpreted as the decay of the iterates of the transfer operator on functions of zero (Lebesgue) average, and this fact is better known as {\em loss of memory.} In the successor paper \cite{NTV15}, a (non-stationary) central limit theorem was  shown for sums of centred observables and with respect to the Lebesgue measure.

We continue here the statistical analysis of these indifferent transformations by proving   the existence of limiting extreme value distributions under suitable normalisation for the threshold of the exceedances.

\section{Conditions for the existence of extreme value laws for non-stationary processes}

%In this section will try to keep as much as possible the notations used in \cite{H83,H86,FFT15}.

In this section, we revise the general theory developed in \cite{FFV16} in order to prove the existence of EVL for non-stationary processes, which is particularly suitable for application to  processes arising from non-autonomous systems. However, since in our application there is no clustering of exceedances, we simplify the exposition by adapting the general conditions and setting to this particular case of absence of clustering.

Let $X_0, X_1, \ldots$ be a stochastic process\index{random/stochastic process}, where each r.v. $X_i:\Y\to\R$ is defined on the measure space $(\Y,\mathcal B,\p)$. We assume that $\Y$  is a sequence space %with a natural product structure so 
such that each possible realisation of the stochastic process\index{random/stochastic process} corresponds to a unique element of $\Y$ and there exists a measurable map $\TF:\Y\to\Y$, the time evolution map, which can be seen as the passage of one unit of time, so that
%\[
$X_{i-1}\circ \TF =X_{i}, \quad \mbox{for all $i\in\N$}.$
%\]
The $\sigma$-algebra $\mathcal B$ can also be seen as a product $\sigma$-algebra adapted to the $X_i$'s.
For the purpose of this paper, $X_0, X_1,\ldots$ is possibly non-stationary. Stationarity would mean that $\p$ is $\TF$-invariant. Note that $X_i=X_0\circ \TF_i$, for all $i\in\N_0$, where $\TF_i$ denotes the $i$-fold composition of $\TF$, with the convention that $\TF_0$ denotes the identity map on $\Y$. In the application  below to sequential dynamical systems, we will have that $\TF_i=T_i\circ \ldots\circ T_1$ will be the concatenation of $i$ possibly different transformations $T_1, \ldots, T_i$.

Each random variable $X_i$ has a marginal distribution function (d.f.) denoted by $F_i$, \ie $F_i(x)=\p(X_i\leq x)$. Note that the $F_i$, with $i\in\N_0$, may all be distinct from each other. For a d.f. $F$ we let $\bar F=1-F$. We define $u_{F_i}=\sup\{x:F_i(x)<1\}$ and let $F_i(u_{F_i}-):=\lim_{h\to 0^+} F_i(u_{F_i}-h) =1$ for all $i$.

Our main goal is to determine the limiting law of
$$\mathbf{P}_n=\p(X_0\leq u_{n,0},X_1\leq u_{n,1},\ldots,X_{n-1}\leq u_{n,n-1})$$
as $n\to\infty$, where $\{u_{n,i},i\leq n-1,n\geq 1\}$ is considered a real-valued boundary.
We assume throughout the paper that
\begin{equation}
\bar F_{\max}:=\max\{\bar F_i(u_{n,i}), i\leq n-1\}\to0 \mbox{ as }n\to\infty,
\label{Fmax}
\end{equation}
which is equivalent to
$u_{n,i}\to u_{F_i}\mbox{ as }n\to\infty,\mbox{ uniformly in } i.$
Let us denote $F^*_n:=\sum_{i=0}^{n-1}\bar F_i(u_{n,i}),$
and assume that  there is $\tau>0$ such that
\begin{equation}
F^*_n:=\sum_{i=0}^{n-1}\bar F_i(u_{n,i})\to\tau, \qquad \mbox{as $n\to\infty.$}
\label{F_n}
\end{equation}

In what follows, for every $A\in\mathcal B$, we denote the complement of $A$ as $A^c:=\mathcal Y\setminus A$.
Let $\mathbb A:=(A_0,A_1,\ldots)$ be a sequence of events such that $A_i\in \TF_i^{-1} \mathcal B$. For some $s,\ell\in \N_0$, we define
\begin{equation}
\label{eq:W-def}
\mathscr W_{s,\ell}(\mathbb A)=\bigcap_{i=s}^{s+\ell-1} A_i^c.
\end{equation}

We will write $\mathscr W_{s,\ell}^c(\mathbb A):=(\mathscr W_{s,\ell}(\mathbb A))^c$. We consider
$\mathbb A_n^{(0)}:=(A_{n,0}^{(0)},A_{n,1}^{(0)},\ldots),$ where the event $A_{n,i}^{(0)}$ is defined as
$A_{n,i}^{(0)}(u_{n,i}):=\{X_i>u_{n,i}\}$.

Now, we recall a mixing condition, introduced in \cite{FFV16}, which was specially designed for the application to the dynamical setting.

\begin{condition}[$\D_0(u_{n,i})$]\label{cond:D} We say that $\D_0(u_{n,i})$ holds for the sequence $X_0,X_1,\ldots$ if for every  $\ell,t,n\in\N$,
\begin{equation}\label{eq:D1}
\left|\p\left(\A_{n,i}\cap
 \mathscr W_{i+t,\ell}\left(\mathbb A_n^{(0)}\right) \right)-\p\left(\A_{n,i}\right)
  \p\left(\mathscr W_{i+t,\ell}\left(\mathbb A_n^{(0)}\right)\right)\right|\leq \gamma_i(n,t),
\end{equation}
where $\gamma_i(n,t)$ is decreasing in $t$ for each $n$ and each $i$ and there exists a sequence $(t_n^*)_{n\in\N}$ such that $t_n^* \bar F_{\max}\to0$ and
$\sum_{i=0}^{n-1}\gamma_i(n,t_n^*)\to0$ when $n\rightarrow\infty$.
\end{condition}

In order to prove the existence of a distributional limit for $\mathbf P_n$, in \cite{FFV16}, we used as usual a blocking argument that splits the data into $k_n$ blocks separated by time gaps of size larger than $t_n^*$, which are created by simply disregarding the observations in the time frame occupied by the gaps. The precise construction of the blocks is given in \cite[Section~2.2]{FFV16} but we briefly describe here some of the key properties of this construction.

In the stationary context, one takes blocks of equal size, which in particular means that the expected number of exceedances within each block is $n\p(X_0>u_n)/k_n\sim \tau/k_n$. Here the blocks may have different sizes, which we denote by $\ell_{n,1}, \ldots, \ell_{n, k_n}$ but, as in \cite{H83,H86}, these are chosen so that the expected number of exceedances is again $\sim\tau/k_n$.  Also, for $i=1,\ldots,k_n$, let $\L_{n,i}=\sum_{j=1}^{i}\ell_{n,j}$ and $\L_{n,0}=0$. See beginning of Section~2.2 of \cite{FFV16} for the precise definition of these quantities.

We recall now a condition that imposes some restrictions on the speed of recurrence within each block, which, in the present context, precludes the existence of clustering.

%For $q\in\N_0$ given by \eqref{def:q},
Consider the sequence $(t_n^*)_{n\in\N}$, given by condition  $\D_0(u_{n,i})$ and let $(k_n)_{n\in\N}$ be another sequence of integers such that
\begin{equation}
\label{eq:kn-sequence}
k_n\to\infty\quad \mbox{and}\quad  k_n t_n^* \bar F_{\max}\to0,\quad \text{as $n\rightarrow\infty$.}
\end{equation}

\begin{condition}[$\D'_0(u_{n,i})$]\label{cond:D'q} We say that $\D'_0(u_{n,i})$
holds for the sequence $X_0,X_1,X_2,\ldots$ if there exists a sequence $(k_n)_{n\in\N}$ satisfying \eqref{eq:kn-sequence} and such that
\begin{equation}
\label{eq:D'rho-un}
\lim_{n\rightarrow\infty}\sum_{i=1}^{k_n} \sum_{j=0}^{\ell_i-1} \sum_{r=j+1}^{\ell_i-1}\p(A^{(0)}_{\L_{i-1}+j}\cap A^{(0)}_{\L_{i-1}+r})=0.
\end{equation}
\end{condition}

%When $q=0$, as we will see to be the case here,
Condition $\D'_0(u_{n,i})$ precludes the occurrence of clustering of exceedances.

The following is a corollary of  \cite[Theorem~2.4]{FFV16}, in the particular case of absence of clustering and which we will use below to obtain the existence of EVL.

\begin{theorem}
\label{thm:error-terms-general-SP-no-clustering}
Let $X_0, X_1, \ldots$ be a stationary stochastic process and suppose \eqref{Fmax} and \eqref{F_n} hold for some $\tau>0$. %Let $q\in\N_0$ be as in \eqref{def:q} and assume that \eqref{eq:EI} holds.
Assume that conditions $\D_0(u_{n,i})$ e $\D'_0(u_{n,i})$ are satisfied. Then
\begin{align*}
\lim_{n\to \infty}\mathbf P_n=e^{- \tau}.
\end{align*}
\end{theorem}

\section{Sequential systems on intermittent maps: statement of the main result}

We consider maps with indifferent fixed points in the formulation proposed  in \cite{LSV99}. Namely, for $\alpha\in (0,1)$,
\begin{equation}
\label{eq:int-family}
T_\alpha(x)=\begin{cases} x(1+2^\alpha x^\alpha) & \text{ for } x\in [0, 1/2)\\
2x-1 & \text{ for } x\in [1/2, 1]\end{cases}
\end{equation}
and we concatenate them. For each $i\in\N$, let $T_i=T_{\alpha_i}$, with $\alpha_i\in(0,\alpha^*)$, where $\alpha^*=1/7$.
% As before, $\TF_n=T_n\circ\ldots\circ T_1$.

This countable sequence of maps $\{T_i\}_{i\in\N}$  allows us to define a {\em sequential dynamical system}. A {\em sequential orbit} of $x\in X$  will be defined by the concatenation
\begin{equation}\label{m}
\TF_n (x) :=T_n\circ\cdots\circ T_1 (x), \ n\ge 1.
\end{equation}
We  denote by $P_{j}$ the Perron-Fr\"obenius (transfer) operator associated to $T_{j}$ defined by the duality relation
$$
\int_X P_{j}f \ g\ dm\ = \ \int_X f\ g\circ T_{j} \ dm,\;\; \mbox{ for all } f\in L^1_m, \ g\in L^{\infty}_m.
$$
Note that here the transfer operator $P_{j}$ is defined with respect to the reference Lebesgue measure $m$.

Similarly to (\ref{m}), we define the composition of operators as
\begin{equation}\label{o}
\Pi_n:=P_n\circ\cdots\circ P_1, \ n\ge 1.
\end{equation}
It is easy to check that duality persists under concatenation, namely
\begin{multline*}%\label{c}
\int_X g  (\TF_n) \  f \ dm=\int_X g (T_n\circ\cdots\circ T_1) \ f\ dm=\\ \int_X g(\ P_n\circ\cdots\circ P_1f )\ dm\ =  \int_X g\ (\Pi_n f) \ dm.
\end{multline*}

 We note that this perturbation {\em by changing the slope} has also been considered for other interesting purposes. The first result, by Freitas and Todd \cite{FT} is about {\em statistical stability}, which establishes the continuity in $L^1$ of the densities of the absolutely continuous invariant measures when the parameter $\alpha$ changes. A strong achievement in this direction has been obtained, independently, by Baladi and Todd \cite{BT16}, Korepanov \cite{K16} and, more recently, Bahsoun and Saussol \cite{BS16}, with the proof of the differentiability of the function $\alpha \rightarrow \int \psi d\mu_{\alpha}$, where $\mu_{\alpha}$ is the absolutely continuous invariant measure for $T_{\alpha}$ and $\psi$ is a function in some $L^q$; we defer to those papers for the precise definition and for the differences among them. We just stress that as a consequence, it is possible to obtain linear response and, in particular, \cite{BT16} gives a
formula for the value of the derivative. 

Let us now focus on  the situation of our interest, namely the sequential or random composition of these kind of maps. Whenever a finite number of them are chosen in an i.i.d. way and with a position dependent probability distribution $\mathbb{P}$, the stochastic stability was proven by Duan \cite{DU}. Still in this framework and by considering the {\em annealed situation} where the statistics is insured by the direct product of $\mathbb{P}$ with the stationary measure, Bahsoun, Bose and Duan \cite{BBD} proved polynomial  decay of correlations, and successively Bahsoun and Bose \cite{BB} got a central limit theorem. The latter was successively generalized in the {\em quenched} case (with respect to the stationary measure and for almost all the realizations), by Nicol, Torok and Vaienti \cite{NTV15}; this paper contains also a proof of the central limit theorem for sequential systems and its results will be used again in the next section. Still in this context we also quote the paper by  Lepp\"anen and Stenlund \cite{LS16} where a few results on the continuity of the densities and their pushforward with respect to  the parameter $\alpha$ are proved.

We now turn to the context of extreme value analysis. Similarly to \cite{FFT10} (in the context of stationary deterministic systems), we consider that the time series $X_0, X_1,\ldots$ arises from these sequential systems simply by evaluating a given observable $\varphi:X\to\R\cup\{\pm\infty\}$ along the sequential orbits,
\begin{equation}
\label{eq:def-stat-stoch-proc-DS} X_n=\varphi\circ \TF_n,\quad \mbox{for
each } n\in {\mathbb N}.
\end{equation}
Note that, on the contrary to the setup in \cite{FFT10},   the stochastic process $X_0, X_1,\ldots$ defined in this way is not necessarily stationary, because $m$ is not an invariant measure for any of the $T_i$.

We assume that the r.v. $\varphi:X\to\R\cup\{\pm\infty\}$
achieves a global maximum at $\zeta\in [0,1]$ (we allow
$\varphi(\zeta)=+\infty$) being of following form:
\begin{equation}
\label{eq:observable-form}
\varphi(x)=g\big(\dist(x,\zeta)\big),
\end{equation} where $\zeta$ is a chosen point in the
phase space $[0,1]$ and the function $g:[0,+\infty)\rightarrow {\mathbb
R\cup\{+\infty\}}$ is such that $0$ is a global maximum ($g(0)$ may
be $+\infty$); $g$ is a strictly decreasing continuous bijection $g:V \to W$
in a neighbourhood $V$ of
$0$; and has one of the three types of behaviour described in equations (1.11), (1.12) and (1.13) of \cite{FFT10}, which are important to determine the type of EVL that applies under linear normalisation (see \cite[Remark~6]{F13a}).

We now choose {\em time-dependent levels} $u_{n,i}$  given by $m(X_i>u_{n,i})=\tau/n$, where $\tau\geq0$. Let $\delta_{n,i}=g^{-1}(u_{n,i})$ so that
\begin{equation}
\label{eq:uni-intermmitent}
m(X_i>u_{n,i})=\int \I_{(\zeta-\delta_{n,i},\zeta+\delta_{n,i})}\Pi_i(1)dm=\frac\tau n.
\end{equation}
Observe that $\delta_{n,0}=\frac\tau{2n}$ and, by Lemma~\ref{bound-Pi}, which appears below, for $n$ sufficiently large, we have that for some constants $0<c<C'$,
\begin{equation}
\label{eq:deltan-estimates}
\frac\tau{2C'n}\leq\delta_{n,i}\leq \frac\tau{2cn}.
\end{equation}

Note that this choice for the levels $u_{n,i}$ guarantees that condition \eqref{F_n} is trivially satisfied.

We are now in condition of stating and proving our main result.

\begin{theorem}
\label{thm:intermittent} Consider the family of maps given by \eqref{eq:int-family} and the sequential dynamics given by $\TF_n=T_n\circ\ldots\circ T_1$, where $T_i=T_{\alpha_i}$, with $\alpha_i\in(0,\alpha^*)$ and $\alpha^*=1/7$. Let $X_1, X_2,\ldots$ be defined by \eqref{eq:def-stat-stoch-proc-DS}, where the observable function $\varphi$, given by \eqref{eq:observable-form}, achieves a global maximum at a chosen $\zeta\in(0,1]$. For $m$-a.e. $\zeta\in(0,1]$, we may define the levels $(u_{n,i})_{n,i\in\N}$ such that \eqref{eq:uni-intermmitent} holds for some $\tau\geq0$, conditions $\D_0(U_{n,i})$ and $\D'_0(U_{n,i})$ hold and consequently:
$$
\lim_{n\to\infty}m(X_0\leq u_{n,0},X_1\leq u_{n,1},\ldots,X_{n-1}\leq u_{n,n-1})=\e^{-\tau}.
$$
\end{theorem}

\begin{remark}
\label{rem:alpha}
We emphasise that this restriction on $\alpha$ ($\alpha<1/7$) is rather technical and is due to the use of the blocking argument and of decay of correlations, which is proved only on sufficiently regular Banach spaces of functions. We remark that the same techniques gave rise to similar restrictions on $\alpha$ even in the stationary setting, where the orbits are obtained by iterations of the same Liverani-Saussol-Vaienti map (see \cite[Section~3.4]{HNT12}). It is interesting to observe that the threshold value $\alpha<1/7$ is the same appearing in \cite{NTV15} in order to establish the central limit theorem for smooth observable.
\end{remark}

\section{Proof of the theorem}
By Theorem~\ref{thm:error-terms-general-SP-no-clustering}, to prove Theorem \ref{thm:intermittent} we only need to check conditions $\D_0(u_{n,i})$ and $\D'_0(u_{n,i})$.
\subsection{Verification of  $\D_0(u_{n,i})$}
\label{subsec:D0}
The intermittent map introduced above exhibits polynomial decay of correlations, which can be obtained by considering  decay of the $L^1$ norm of the concatenation of the Perron-Frobenius operators: this fact is also known as {\em loss  of memory}.  We will be interested in the  kind of correlations given
 in \cite[Proposition~4.3]{FFV16}, %Proposition \ref{prop:decay-correlations},
 which reads
\begin{align*}
D\!C&(\phi,\psi,i,t):=\int \phi\circ \TF_i\psi \circ \TF_{i+t}dm-\int \phi\circ \TF_i dm\int  \psi \circ \TF_{i+t} dm \\&=  \int\left(\psi-\int\psi \Pi_{i+t}(1)dm\right)\,P_{i+t}\ldots P_{i+1}\left(\Pi_i(1)\left(\phi-\int\phi \Pi_{i}(1)dm\right)\right).
\end{align*}
Let $\tilde\phi= \phi-\int\phi \Pi_{i}(1)dm$. Observe that $\int \Pi_i(1)\tilde \phi dm=0.$ This means that the observable function $\Pi_i(1)\tilde \phi\in \mathcal V_0$, where $\mathcal V_0$ is the set of functions with 0 integral that was defined in \cite[Lemma~2.12]{CR07}.\\ Now, contrary to what we did in the case of uniformly expanding maps, we will consider decay of the $L^1$ norm of the concatenation of the PF operators, namely we will consider, having set $\tilde\psi= \psi-\int\psi \Pi_{i}(1)dm:$

\begin{align}\label{DCC}
|D\!C(\phi,\psi,i,t)|&=\left|\int\tilde\psi\,P_{i+t}\ldots P_{i+1}\left(\Pi_i(1)\tilde\phi\right)dm\right|\\
	&\leq \|P_{i+t}\ldots P_{i+1}(\Pi_i(1)\tilde \phi)\|_{1}\ ||\psi||_{\infty}.
\end{align}

To deal with such correlations we apply the following result proved in \cite{AHNTV15}:
\begin{theorem}[\cite{AHNTV15}]\label{thm:decay}

  Suppose $\psi, \phi$ are in the cone ${\mathcal C}_a$ (see below), for some $a$ and with equal
  expectation $\int \phi dm= \int \psi dm$. Then for any $0<\alpha^*<1$ and for
  any sequence $T_{1},\cdots, T_{n}$, $n \ge 1$, of maps of
  Pomeau-Manneville type
  with $0 < \alpha_k\le \alpha^* < 1$, $k\in
  [1,n]$, we have
  \begin{equation}\label{DC}
    \int |\Pi_{n}(\phi)-\Pi_{n}(\psi)|dm \le C_{\alpha^*}
    (\|\phi\|_1+\|\psi\|_1)n^{-\frac{1}{\alpha}+1}(\log n)^{\frac{1}{\alpha}},
  \end{equation}
  where the constant $C_{\alpha^*}$ depends only on the map $T_{\alpha^*}$.

\end{theorem}
The cone  $\mathcal{C}_a$ contains functions given by
(here $X(x)=x$ denotes the identity function):
\[
\mathcal{C}_a= \{f\in C^0((0,1])\cap L^1(m) \mid \ f\ge 0, \ f \
\mbox{decreasing},\ X^{\alpha+1}f \ \mbox{increasing}, \ f(x)\le a x^{-\alpha}\int f dm
\}
\]
Having fixed $0< \a < 1$, it was  proven in~\cite{AHNTV15} that, provided $a$ is large
enough, the cone $\mathcal{C}_a$ is preserved by all operators
$P_{k}$.\\

We are now ready  to verify  $\D_0(u_{n,i})$.  %since $q=0$ and we are taking $u_{n,i}=u_n$
Note that $A_{n,i}^{(0)}=\{X_i>u_{n,i}\}=:U_{n,i}$ is an interval.

We will apply the bound (\ref{DCC}). We begin to observe that  in our case $\phi$ is not in the cone $\mathcal{C}_{a}$; we therefore approximate it with a function $\chi$ which is $C^1$ and with compact support, equal to $1$ on $U_{n,i}$ and rapidly decreasing to zero on a set $\Lambda$ of diameter $\Delta$ in the complement of $U_{n,i}$. \footnote{This can be achieved for instance in this way. Let $U_n=(a_n,b_n)$ and $U_n^\Delta=(a_n-\Delta,b_n+\Delta)$.
Define
\begin{equation*}
\chi(x)=\begin{cases} 1 & \text{ for } x\in (a_n, b_n)\\
\e^{-\frac1{1-\left(\frac{x-b_n}{\Delta}\right)^2}} & \text{ for } x\in [b_n,b_n+\Delta)\\
\e^{-\frac1{1-\left(\frac{x-a_n}{\Delta}\right)^2}} & \text{ for } x\in (a_n-\Delta,a_n]\\
0 & \text{ for } x\in \R\setminus U_n^\Delta
\end{cases}.
\end{equation*}

Note that $\Delta U_n:=\{x:\;\chi(x)-\I_{U_n}(x)>0\}=U_n^{\Delta}\setminus [a_n,b_n]$ and $m(\Delta U_n)=2\Delta$. We have $\chi\in \mathcal C^\infty$, $\chi''(b_n+\frac\Delta{3^{1/4}})=0=\chi''(a_n-\frac\Delta{3^{1/4}})$ and
$$
\max\{\chi'(x)\}=\chi'(b_n+\frac\Delta{3^{1/4}})=\chi'(a_n-\frac\Delta{3^{1/4}})=\frac{2\e^{-\frac1{1-1/\sqrt3}}}{3^{1/4}(1-1/\sqrt3)^2}\frac1\Delta=O(1/\Delta).
$$.
} We have that $||\chi||_{\infty}=1$, $||\chi'||_{\infty}=O(\Delta^{-1})$ and finally $||\phi-\chi||_1=O(\Delta).$ 
In this way we have:
$$
\Pi_i(1)\tilde \phi=\Pi_i(1) \chi-\Pi_i(1)\int \chi \Pi_i(1)dm + \Pi_i(1)[\phi-\chi]-\Pi_i(1)\int [\phi-\chi]\Pi_i(1)dm.
$$
To this quantity we have to apply the power $\Pi_t:=P_{i+t}\ldots P_{i+1}$ and then take  the $L^1$ norm: for the  last two terms in the preceding identity this contribution will be of order $2\Delta.$ 
Now, generalizing an argument in  \cite{LSV99}, it can be shown, as in \cite{NTV15}, that there are constants $\lambda<0, \nu>0, \delta>0$ such that, having set $\chi':= \chi-\int \chi \Pi_i(1)dm,$ the functions
$$
F:=\chi'\Pi_i(1)+\lambda X \Pi_i(1)+\nu\Pi_i(1)+\delta; \ G:=\lambda X \Pi_i(1)+\nu\Pi_i(1)+\delta
$$
are {\em pushed into the cone} $\mathcal{C}_a$,  in such a way that
$$
\Pi_t(\Pi_i(1)\chi')=\Pi_t(F)-\Pi_t(G),
$$
and, by the above theorem on loss of memory,
$$
||\Pi_t(\Pi_i(1)\chi'||_1=||\Pi_t(F)-\Pi_t(G)||_1\le C_{\a^*}
    (\|F\|_1+\|G\|_1)t^{-\frac{1}{\a^*}+1}(\log t)^{\frac{1}{\a^*}}.
$$
It is important to notice that the constants $\lambda, \nu, \delta$
\begin{itemize}
\item are independent of $i$;
\item are affine functions of the $C^1$ norm of $\chi$, with multiplicative constants depending only on $\a^*.$

\end{itemize}
In conclusion, this means that we can write
$$
||\Pi_t(\Pi_i(1)\chi'||_1\le C_{\a^*}[A_{\a^*}||\chi||_{\infty}+B_{\a^*}||\chi'||_{\infty}+D_{\a^*}]t^{-\frac{1}{\a^*}+1}(\log t)^{\frac{1}{\a^*}},
$$
where the factors $A_{\a^*}, B_{\a^*}, D_{\a^*}$ depend only on $\a^*.$
Therefore, and taking into account the bounds on $\chi$, there will be new constants $C_1, C_2, C_3$ depending only on $\a^*$ such that
$$
||\Pi_t(\Pi_i(1)\tilde{\phi}||_1\le 2\Delta \ + C_1 t^{-\frac{1}{\a^*}+1}(\log t)^{\frac{1}{\a^*}}\ + C_2 \Delta^{-1} t^{-\frac{1}{\a^*}+1}(\log t)^{\frac{1}{\a^*}}\ + C_3 t^{-\frac{1}{\a^*}+1}(\log t)^{\frac{1}{\a^*}}.
$$
Returning to \eqref{DCC}, it follows that there exists $C^*$ (depending only on $\alpha^*$) such that
\begin{equation}
\label{DCC-final-estimate}
D\!C(\phi,\psi,i,t)\leq \left(2\Delta+C^*\Delta^{-1} t^{-\frac{1}{\a^*}+1}(\log t)^{\frac{1}{\a^*}}\right)\|\psi\|_\infty.
\end{equation}
In order to verify condition $\D_0({u_n,i})$, we let $\Delta=n^{1+\eta}$, for some $\eta>0$, $t_n=n^{\kappa}$, for some $0<\kappa<1$ and for each $n,i,\ell$ set $\phi_i=\I_{(\zeta-\delta_{n,i},\zeta+\delta_{n,i})}$ and $\psi_i=\I_{(\zeta-\delta_{n,i+t_n},\zeta+\delta_{n,i+t_n})}\cdot\ldots\cdot \I_{(\zeta-\delta_{n,i+t_n+\ell},\zeta+\delta_{n,i+t_n+\ell})}\circ(T_{i+t_n+\ell}\circ\ldots\circ T_{i+t_n+1})$. Then we can write:
$$
D\!C(\phi_i,\psi_i,i,t_n)\leq 2n^{-(1+\eta)}+C^*n^{1+\eta}n^{(-\frac1{\alpha^*}+1)\kappa}(\kappa\log n)^{\frac1{\alpha^*}}=:\gamma_i(n,t_n).
$$
Then, for some $C^{**}>0$, we have
$
\sum_{i=0}^{n-1}\gamma_i(n,t_n)\leq 2n^{-\eta}+C^{**}n^{2+2\eta}n^{(-\frac1{\alpha^*}+1)\kappa}\to0,\;\text{as $n\to\infty$},
$
as long as $\alpha$ is sufficiently small so that $(-\frac1{\alpha^*}+1)\kappa+2+2\eta<0$, which ultimately settles condition $\D_0(u_{n,i})$.

Note that in order to optimise the choice of the $\alpha^*$ (which we want as large as possible), we need to choose $\eta$ close to $0$ and $\kappa$ close to $1$, which means that $\alpha^*<\frac13$. However, in order to prove $\D'_0(u_{n,i})$ we still need further restrictions on $\alpha$.

\subsection{Verification of $\D'_0(u_{n,i})$}
\label{subsec:Dlinha0}

We will begin with a lemma that adjusts to the sequential setting the argument used in \cite[Lemma~3.10]{HNT12}. Essentially, it says that the Lebesgue measure of the points that after $n$ iterations by the sequential intermittent maps return to an $\eps$ neighbourhood of themselves scales like a power of $\eps$ that depends on $\alpha^*$. 

Let $\mathcal E_n(\eps):=\{x\in[0,1]:\;|\TF_n(x)-x|\leq \eps\}.$

\begin{lemma}
\label{lem:aux-Dlinha}
There exists some $C>0$ such that for all $n\in\N$, we have $$m(\mathcal E_n(\eps))\leq C% \min\{\eps^{1/2},
\eps^{1/(1+\alpha^*)}.$$
\end{lemma}

\begin{proof}
Let $J_1, J_2, \ldots, J_k$ be the domains of injectivity of $\TF_n$, ordered from the left to the right, \ie $J_i=[a_i,b_i)$ and $0=a_1<b_1=a_2<\ldots<b_{k-1}=a_k<b_k=1$. Note that $\TF_n$ is full branched map, in particular, each branch $\TF_n|_{J_i}$ is a convex map where for each $i\neq 1$ we have $D\TF_n(x)>\gamma>1$ but when $i=1$, we have $D\TF_n(0)=1$.

We consider now an $\eps$-neighbourhood of the diagonal and the intersection of its boundary with the full branches of $\TF_n$, \ie we define for each $i=1,\ldots, k$, the points $x_i^\pm\in J_i$ such that $\TF_n(x_i^\pm)=x_i^\pm\pm\eps$, whenever this intersection is well defined. Note that, whenever both points $x_i^\pm$ exist then $\mathcal E_n(\eps)\cap J_i\subset [x_i^-,x_i^+]$.

Let $x\geq x_i^-$ in $J_i$. By convexity of $\TF_n|_{J_i}$, we have
$
D\TF_n(x)\geq D\TF_n(x_i^-)\geq\frac{x_i^--\eps-\TF_n(a_i)}{x_i^--a_i},
$
hence
$
D\TF_n(x)-1\geq \frac{x_i^--\eps-\TF_n(a_i)}{x_i^--a_i}-1=\frac{a_i-\eps-\TF_n(a_i)}{x_i^--a_i}\geq \frac{a_i-\eps-\TF_n(a_i)}{m(J_i)}.
$
It follows that
$
2\eps=\int_{x_i^-}^{x_i^+} D\TF_n(x)-1 dx\geq m([x_i^-,x_i^+])\frac{a_i-\eps-\TF_n(a_i)}{m(J_i)},
$
which implies
$$
\mathcal E_n(\eps)\cap J_i\leq \frac{2\eps}{a_i-\eps-\TF_n(a_i)}m(J_i).
$$
This estimate is useful whenever $a_i-\eps-\TF_n(a_i)$ is not small. Hence, we define
$$
V^\eta=\cup\{a_i:\;|a_i-\TF_n(a_i)|<\eps+\eta\}\quad\mbox{and}\quad Z^\eta=\cup_{a_i\in V^\eta}J_i.
$$
Then
$
m(\mathcal E_n(\eps))=m(\mathcal E_n(\eps)\cap Z^\eta)+m(\mathcal E_n(\eps)\cap (Z^\eta)^c)\leq
m(Z^\eta)+\frac{2\eps}{\eta}m((Z^\eta)^c).
$

Now we estimate these sets in two different ways depending on whether $n$ is small or large. 

Assume that $\eps<\eta$ and $n$ is sufficiently large so that $\max_i |J_i|\leq \eps$, where $|J_i|=b_i-a_i$. Recall that $\TF_n(a_i)=0$ for all $i$. Since $a_i\in V^\eta$ means that $a_i<\eta+\eps$ then
$
m(\mathcal E_n(\eps))\leq 2 \eta+\frac{2\eps}{\eta}.
$
Optimising over $\eta\in (0,1)$ we have that $\eta=O(\sqrt \eps)$ is the best choice and gives $m(\mathcal E_n(\eps))\leq C \sqrt \eps\leq C\eps^{1/(1+\alpha^*)},$
since as mentioned above we have $\alpha^*<1/2$, which implies that $1/(1+\alpha^*)>2/3>1/2$.

When $n$ is small then the worst case scenario happens on $J_1$. In this case $x_1^-$ is not defined and $\mathcal E_n(\eps)\cap J_1=[0,x_1^+]$. In this case, we have:
$
\eps=\TF_n(x_1^+)-x_1^+\geq T_{\alpha^*}(x_1^+)-x_1^+=2^{\alpha^*}(x_1^+)^{1+\alpha^*},
$
which implies that $x_1^+=\left(\frac\eps{2\alpha^*}\right)^{\frac1{1+\alpha^*}}$ and ultimately, for $\alpha\in(0,1)$, taking $\eta=\sqrt\eps$, we have $m(\mathcal E_n(\eps))\leq \eps^{\frac1{1+\alpha^*}}$.
\end{proof}

We now follow the argument originally used by Collet in \cite{C01} and further developed in \cite{HNT12}. Let $0<\beta<1$, $0<\kappa<\beta$ and $0<\xi<1$ such that $\kappa(1+\xi)<\beta$. We define the set of points that recur too fast:
$$
E_j=\left\{x\in[0,1]: \; |\TF_i(x)-x|\leq \frac 2 j\;\mbox{for some}\; i\leq j^{\kappa(1+\xi)}\right\}.
$$

By Lemma~\ref{lem:aux-Dlinha}, we have that
$
m(E_j)\leq \sum_{i=1}^{j^{\kappa(1+\xi)}}m(\mathcal E_i( 2/ j))\leq \frac C {j^\varsigma},
$
where $\varsigma=\frac{1}{1+\alpha^*}-\kappa(1+\xi)$ and for some $C>0$.

The core of Collet's argument is based on the use of Hardy-Littlewood maximal functions to obtain, from an estimate on the measure of the sets $E_j$, an estimate for the conditional measure on balls  of radius $1/j$, centred on $m$-a.e point $\zeta$,  of the intersection of these sets $E_j$ with the corresponding balls.

\begin{lemma}
\label{lem:Collet}
Assume that $(E_n)_{n\in\N}$ is a sequence of measurable sets such that
$$m(E_j)\leq \frac C{j^\varsigma},$$
for some $C,\varsigma>0$. Then for $0<\beta<\varsigma$ and $\gamma>1/(\varsigma-\beta)$, we have that for $m$-a.e. $\zeta\in[0,1]$, there exists $N(\zeta)$ such that for all $j\geq N(\zeta)$
$$m(\{|x-\zeta|\leq j^{-\gamma}\}\cap E_{j^\gamma})\leq \frac2{j^{\gamma+\gamma\beta}}.$$
\end{lemma}
\begin{proof}
Define the Hardy-Littlewood maximal function:
$$
L_n(x)=\sup_{\ell>0}\frac1{2\ell}\int_{x-\ell}^{x+\ell}\I_{E_n}(z)dz.
$$
By the Theorem of Hardy-Littlewood we have
$
m(L_n>\lambda)\leq \frac C\lambda \|\I_{E_n}\|_{L^1}=\frac C\lambda m(E_n).
$
Taking $\lambda=n^{-\beta}$ with $0<\beta<\varsigma$, we have
$
m(L_n>n^{-\beta})\leq \frac{c}{n^{-\beta}}m(E_n)\leq\frac C{n^{\varsigma-\beta}}.$
Hence, taking $n=j^\gamma$, we have
$m(L_{j^\gamma}>j^{-\beta\gamma})\leq \frac C{j^{\gamma(\varsigma-\beta)}}$ and assuming that $\gamma(\varsigma-\beta)>1$ it follows that
$
\sum_{j}m(L_{j^\gamma}>j^{-\beta\gamma})\leq \sum_j \frac C{j^{\gamma(\varsigma-\beta)}}<\infty.
$
Hence, by the Borel-Cantelli lemma we have that
for $m$-a.e. $\zeta$ there exists $N(\zeta)$ such that for all $j\geq N(\zeta)$ we have $\zeta\in \{L_{j^\gamma}\leq j^{-\beta\gamma}\}$.Choosing $\ell=j^{-\gamma}$, by definition of the function $L$, we have for $m$-a.e. $\zeta$
$$
\int_{x-\ell}^{x+\ell}\I_{E_n}(z)dz=m((\zeta-j^{-\gamma}, \zeta+j^{-\gamma})\cap E_{j^\gamma})\leq 2 j^{-\gamma(1+\beta)}.
$$
\end{proof}

\begin{lemma}
\label{bound-Pi}
There exist constants $c, C, C', C''>0$ such that for all $i\in\N$ and $x\in[0,1]$ we have
$$
c\leq \Pi_i(1)(x)\leq C x^{-\alpha}.
$$
In particular, for $x\in U_n$ and $n$ sufficiently large, taking $C'=C'' \zeta^{-\alpha}$, we can write
$$
c\leq \Pi_i(1)(x)\leq C'.
$$
%where $C'=C'' \zeta^{-\alpha}$.
\end{lemma}
\begin{proof}
It is enough to prove the first inequalities. The upper bound follows because the constant function $1$ is in the cone $\mathcal{C}_a$ and therefore for any $P_i: (P_i 1)(x)\le a x^{\alpha} \int P_i 1 dm\le a x^{\alpha};$ in this case $C=a.$ The lower bound is the content of Lemma 2.4 in \cite{LSV99} with $c= \min\left\{a, \left[\frac{\alpha(1+\alpha)}{a^{\alpha}}\right]^{\frac{1}{1-\alpha}}\right\}.$
\end{proof}

\begin{lemma}
\label{lem:Dlinha-firstestimate}
There exists a constant $C>0$ such that for $m$-a.e. $\zeta\in(0,1]$, for all $\ell\in\N$ and all $n$ sufficiently large, we have
$$
n\sum_{i=1}^{n^\kappa} m\left(\left\{x:\; |\TF_\ell(x)-\zeta|\leq\delta_{n,\ell}\;\mbox{and}\; |\TF_{i+\ell}(x)-\zeta|\leq \delta_{n,i+\ell}\right\}\right)\leq C \frac{n^{\kappa}}{n^{\beta}}\xrightarrow{n\to\infty}0.
$$
%where $\delta_n=\frac\tau{2nh(\zeta)}$ and $h(\zeta)=\frac{d\mu}{dm}(\zeta)$.
\end{lemma}
\begin{proof}
Let $j=\left(\frac{cn}{\tau}\right)^{1/\gamma}$ so that $j^{-\gamma}=\tau/(cn)$. Also observe that $n^\kappa=(\tau j^{\gamma}/c)^\kappa\leq j^{\gamma\kappa(1+\xi)}$, if $n$ is large enough. Hence, for such sufficiently large $n$, we have:
\begin{align*}
V_n:=&\{x:\; |x-\zeta|\leq \frac \tau{cn}\;\mbox{and}\; |\TF_i(x)-\zeta|\leq \frac \tau{cn}\;\mbox{for some $i\leq n^\kappa$}\}\\
%\subset& \{x:\; |x-\zeta|\leq j^{-\gamma}\;\mbox{and}\; |\TF_i(x)-\zeta|\leq j^{-\gamma}\;\mbox{for some $i\leq n^\kappa$}\}\\
\subset & \{x:\; |x-\zeta|\leq j^{-\gamma}\;\mbox{and}\; |\TF_i(x)-x|\leq 2j^{-\gamma}\;\mbox{for some $i\leq n^\kappa$}\}\\
\subset & \{x:\; |x-\zeta|\leq j^{-\gamma}\;\mbox{and}\; |\TF_i(x)-x|\leq 2j^{-\gamma}\;\mbox{for some $i\leq j^{\gamma\kappa(1+\xi)}$}\}\\
=&\{x:\;|x-\zeta|\leq j^{-\gamma}\}\cap E_{j^\gamma}.
\end{align*}
Hence, by Lemma~\ref{lem:Collet} we have $m(V_n)\leq 2\tau^{1+\beta}/n^{1+\beta}$. Taking $C=2\tau^{1+\beta}$, we have
\begin{align}
\label{eq:estimate1-lemDlinha}
n\sum_{i=1}^{n^\kappa} m\left(\left\{x:\; |x-\zeta|\leq\frac \tau{cn},\; |\TF_i(x)-\zeta|\leq \frac \tau{cn}\right\}\right)&\leq n\sum_{i=1}^{n^\kappa} m(V_n)\leq n^{1+\kappa}\frac{2\tau^{1+\beta}}{n^{1+\beta}}\nonumber\\
&\leq C\frac {n^\kappa}{n^\beta}.
\end{align}
Finally, we observe that the quantity we want to estimate can be written as
\begin{multline*}
n\sum_{i=1}^{n^\kappa} \int \I_{B_{\delta_{n,\ell}}(\zeta)}\circ \TF_{\ell} \,\I_{B_{\delta_{n,i+\ell}}(\zeta)}\circ\TF_{i+\ell} dm \\= n\sum_{i=1}^{n^\kappa} \int \I_{B_{\delta_{n,\ell}}(\zeta)}\,\I_{B_{\delta_{n,i+\ell}}(\zeta)}\circ T_{i+\ell}\circ\ldots\circ T_{\ell+1} \,\Pi_\ell(1) dm.
\end{multline*}
Recalling that by \eqref{eq:deltan-estimates} we have $\delta_{n,i}\leq\frac\tau{cn}$, for all $i\in\N_0$, then, by Lemma~\ref{bound-Pi} and \eqref{eq:estimate1-lemDlinha}, it follows that there exist $C', C''>0$ such that
$$
n\sum_{i=1}^{n^\kappa} \int \I_{B_{\delta_{n,\ell}}(\zeta)}\circ \TF_{\ell} \,\I_{B_{\delta_{n,i+\ell}}(\zeta)}\circ\TF_{i+\ell} dm\leq C' n\sum_{i=1}^{n^\kappa} m(V_n) \leq C''\frac {n^\kappa}{n^\beta}.
$$
\end{proof}

Recall that we are taking: $k_n=n^{1-\beta}$ and $t_n=n^{\kappa}$.

From Lemma~\ref{bound-Pi}, we have that $c\mu(U_n)\leq m(X_j>u_n)\leq C\mu(U_n)$. Hence, if we let $L_n=\max\{\ell_{i}:\;i=1, \ldots, k_n\}$, we obtain that there exists a constant $\tilde C>0$ such that $L_n\leq \tilde C n^\beta$.

In order to prove $\D'_0$, we need to control the sum on the left
\begin{align*}\sum_{i=1}^{k_n} \sum_{j=0}^{\ell_i-1} \sum_{r=j+1}^{\ell_i-1}\p(A_{n,\L_{i-1}+j}^{(0)}\cap A_{n,\L_{i-1}+r}^{(0)})&\leq \sum_{i=1}^{k_n}\sum_{j=0}^{L_n-1} \sum_{r=j+1}^{L_n-1}\p(A_{n,\L_{i-1}+j}^{(0)}\cap A_{n,\L_{i-1}+r}^{(0)})\\
&\leq \tilde C n \max_{\ell=1,\ldots, n} \sum_{i=1}^{\tilde Cn^\beta}\int \I_{U_n}\circ \TF_{\ell}\,\I_{U_n}\circ \TF_{i+\ell}dm.
\end{align*}

From Lemma~\ref{lem:Dlinha-firstestimate} we have that
$$
 \lim_{n\to\infty}n \max_{\ell=1,\ldots, n} \sum_{i=1}^{ n^\kappa}\int \I_{B_{\delta_{n,\ell}}(\zeta)}\circ \TF_{\ell} \,\I_{B_{\delta_{n,i+\ell}}(\zeta)}\circ\TF_{i+\ell} dm=0.
$$

Hence we are left to handle $n \max_{\ell=1,\ldots, n} \sum_{i=n^\kappa}^{\tilde C n^\beta}\int \I_{B_{\delta_{n,\ell}}(\zeta)}\circ \TF_{\ell} \,\I_{B_{\delta_{n,i+\ell}}(\zeta)}\circ\TF_{i+\ell}\,dm$ for which we use decay of correlations. Using \eqref{DCC-final-estimate}, we have:
\begin{multline*}
n \max_{\ell=1,\ldots, n} \sum_{i=n^\kappa}^{\tilde C n^\beta}\int \I_{B_{\delta_{n,\ell}}(\zeta)}\circ \TF_{\ell} \,\I_{B_{\delta_{n,i+\ell}}(\zeta)}\circ\TF_{i+\ell}\,dm\\ \leq C(n^{1+\beta}n^{1+\eta}n^{\kappa(1-1/\alpha^*)}\log(n)^{1/\alpha^*}+n^{-(1+\eta)+\beta+1}+n^{-2}).
\end{multline*}

If we take $\eta=2\beta$ then if $\alpha^*$ is sufficiently small it is easy to check that the terms on right vanish as $n\to\infty$.

Now, we focus on a possible upper bound for $\alpha^*$. From the first term on the rhs of the previous equation we have that
\begin{equation}
2+4\beta+\kappa-\kappa/\alpha^*<0 \iff \alpha^*<\frac \kappa {2+4\beta+\kappa}.
\end{equation}
Moreover, in order to be able to apply Lemma~\ref{lem:Collet} we need that $\varsigma>\beta$ which means that
\begin{equation}
\frac1{1+\alpha^*}-\kappa(1+\xi)>\beta \iff \alpha^*<\beta+\kappa(1+\xi)-1.
\end{equation}
Recall that $\kappa(1+\xi)<\beta$ but we are free to choose any $\beta\in(0,1)$. Analysing both the expressions one obtains that the maximum range for $\alpha^*$ occurs for $\beta$ and $\kappa$ as close as possible to $1$, which means that $\alpha^*\leq 1/7$.

\section*{Acknowledgements}
ACMF and JMF were partially supported by FCT project FAPESP/19805/2014 and by CMUP (UID/MAT/00144/2013), which is funded by FCT with national (MEC) and European structural funds through the programs FEDER, under the partnership agreement PT2020. All authors are supported by FCT project PTDC/MAT-CAL/3884/2014, also supported by the same programs. SV was supported by the ANR- Project {\em Perturbations},  by the project {\em Atracci\'on de Capital Humano Avanzado del Extranjero} MEC 80130047, CONICYT, at the CIMFAV, University of Valparaiso and by the project {\em PHYSECO} of the {\em Programme R\'egional MATH-AmSud} between France, Chili and Uruguay. 

%All authors were partially  supported by FCT project PTDC/MAT/120346/2010, which is funded by national and European structural funds through the programs  FEDER and COMPETE.
SV is grateful to  N. Haydn, M. Nicol and A. T\"or\"ok for several and fruitful discussions on sequential systems, especially in the framework of indifferent maps. JMF is grateful to M. Todd for fruitful discussions and careful reading of a preliminary version of this paper.  All authors acknowledge the Isaac Newton Institute for Mathematical Sciences, where this work was initiated during the program Mathematics for the Fluid Earth.

%\bibliographystyle{abbrv}
%\bibliographystyle{amsalpha}
%\bibliographystyle{alpha}
%\bibliographystyle{annotate}
%\bibliographystyle{alphanum}
%\bibliographystyle{amsalpha}
%\bibliographystyle{/Users/jmfreita/MEOCloud/Jorge/Bibliografia/novostyle}
%\bibliographystyle{/Users/amoreira/MEOCloud/Jorge/Bibliografia/novostyle}


\bibliographystyle{amsplain}

\bibliography{Nonstationary.bib}

\end{document}